\documentclass[letterpaper, 10pt, twocolumn, conference]{ieeeconf}
\usepackage[utf8]{inputenc}

\IEEEoverridecommandlockouts
\usepackage{amsmath}
\usepackage{amsthm}
\usepackage{amssymb}
\usepackage{amsfonts}
\usepackage{bm}
\usepackage{cite}
\usepackage{xcolor}
\usepackage{ulem}
\usepackage{subcaption}
\usepackage{graphicx}
\newcommand{\R}[1]{\ensuremath{\mathbb{R}^{#1}}}
\renewcommand{\S}[1]{\ensuremath{\mathbb{S}_{#1}}}

\newcommand{\cJ}{\bm{\mathcal{J}}}
\newcommand{\ubar}{\Bar{u}}
\newcommand{\bfX}{\mathbf{X}}
\newcommand{\bfM}{\mathbf{M}}
\newcommand{\Zplus}{\mathbb{Z}\sub{+}}

\newcommand{\E}[1]{\mathbb{E}\left[#1\right]}
\newcommand{\Cov}[1]{\mathrm{Cov}\left( #1 \right)}
\newcommand{\tr}[1]{\ensuremath{\mathrm{tr} (#1) }}

\newcommand{\bdiag}[1]{\ensuremath{\mathrm{bdiag} (#1)}}

\newcommand{\transpose}{\ensuremath{^{\mathrm{T}}}}
\newcommand{\sub}[1]{\ensuremath{_{#1}}}
\newcommand{\spr}[1]{\ensuremath{^{#1}}}
\newcommand{\inv}[1]{\ensuremath{{#1}^{-1}}}
\newcommand{\eq}{=}

\newtheorem{proposition}{Proposition}
\newtheorem{problem}{Problem}
\theoremstyle{definition}
\newtheorem{remark}{Remark}
\newtheorem{example}{Example}

\renewenvironment{proof}
    {\noindent \textit{Proof.}}
    {\hfill \ensuremath{\blacksquare}}
    
\newenvironment{smat}{\left[ \begin{smallmatrix} }{\end{smallmatrix} \right]}

\title{Covariance Steering of Discrete-Time Linear Systems with Mixed Multiplicative and Additive Noise}
\author{Isin M. Balci \and Efstathios Bakolas \thanks{This research has been supported in part by NSF award CMMI-1937957.
I. M. Balci (PhD student) and E. Bakolas (Associate Professor) are with the Department of Aerospace Engineering and Engineering Mechanics, The University of Texas at Austin, Austin, Texas 78712-1221, USA, Email: isinmertbalci@utexas.edu, bakolas@austin.utexas.edu. }}
\date{August 2022}

\begin{document}

\maketitle
\thispagestyle{empty}
\pagestyle{empty}

\begin{abstract}
    In this paper, we study the covariance steering (CS) problem for discrete-time linear systems subject to multiplicative and additive noise. Specifically, we consider two variants of the so-called CS problem. The goal of the first problem, which is called the exact CS problem, is to steer the mean and the covariance of the state process to their desired values in finite time. In the second one, which is called the ``relaxed'' CS problem, the covariance assignment constraint is relaxed into a positive semi-definite constraint. We show that the relaxed CS problem can be cast as an equivalent convex semi-definite program (SDP) after applying suitable variable transformations and constraint relaxations. Furthermore, we propose a two-step solution procedure for the exact CS problem based on the relaxed problem formulation which returns a feasible solution, if there exists one. Finally, results from numerical experiments are provided to show the efficacy of the proposed solution methods.
\end{abstract}

\section{Introduction}\label{s:introduction}
In this paper, we study the problem of characterizing causal feedback policies for discrete-time stochastic linear systems which steer the state mean and the state covariance to desired values.
This class of problems is known as Covariance Steering (CS) problems in the relevant literature \cite{p:bakolas2018covarianceautomatica, p:chen2015covariance1, p:goldshtein2017covariance}. Typically, CS problems are addressed only for the case of additive noise. By contrast, in this work, we consider discrete-time linear systems which are excited by both additive and multiplicative noise. 
Throughout the paper, we will study two variations of the CS problem with mixed multiplicative and additive noise. 
In the first problem formulation, the main goal is defined as finding a causal policy which will steer the mean and the covariance of the state process to their respective desired values in finite time. 
In the second problem formulation, we keep the hard constraint on the state mean but the state covariance constraint is ``relaxed'' into a positive semi-definite constraint.
We will refer to the former variation of the CS problem as the ``exact'' CS problem and the latter as the ``relaxed'' CS problem.

\textit{Literature Review:} The early attempts to address CS problems were focused on the infinite horizon case for linear time-invariant systems in which the set of assignable covariance matrices can be characterized in terms of linear matrix inequalities (LMI) \cite{p:skelton1987covassignment, p:skelton1992improvedcovariance}. 
More recently, finite horizon CS problems have gained significant attention. 
Unconstrained CS problem formulations with continuous-time linear systems were first addressed in \cite{p:chen2015covariance1, p:chen2015covariance2} whereas the constrained CS problems for discrete-time linear systems are considered in \cite{p:bakolas2018covarianceautomatica, p:okamoto2018covariancechance}.
Soft constrained versions of the CS problems, in which the terminal covariance assignment constraint is replaced by a terminal cost term which corresponds to the (squared) Wasserstein distance between the terminal state distribution and the goal (Gaussian) distribution, are studied in \cite{p:halder2016covariancewasserstein, p:balci2022exactcovariancewasserstein}.
Furthermore, CS problems for partially observable systems are studied in \cite{p:kotsalis2021robustcovariance} in which control policies based on histories of ``purified outputs'' are utilized.
In all of the aforementioned papers, the system model is assumed to be linear and the noise process is assumed to be an additive white noise process. 

The problem of finding stabilizing controllers for linear systems subject to multiplicative noise using LMIs has been studied in \cite{p:el1995statemultiplicative}.
Model Predictive Control (MPC) algorithms for linear systems subject to state and control multiplicative noise have been developed in \cite{p:cannonmultmpc, p:farinaMPClinearmult}.
Estimation and control design problems are studied in \cite{p:todorovmult}. 
More recently, sampling-based methods for learning the optimal state feedback controllers for linear systems subject to multiplicative noise have been proposed in \cite{p:gravell2020learningmultiplicative, p:coppens2022safelearningmultnoise}.

The CS problem with continuous-time dynamics and multiplicative noise is studied in \cite{p:liu2022continuouscovmultiplicative} where a solution based on coupled Riccati equations is obtained.
However, the authors of \cite{p:liu2022continuouscovmultiplicative} consider the case in which the system is only affected by the state multiplicative noise, and the state mean at the initial stage and its desired terminal value are both zero. 
In our work, we consider a more general problem with system dynamics having both state and control multiplicative noise, and nonzero initial and desired mean dynamics. 
To the best of our knowledge, this is the first paper that addresses the finite horizon CS problem for discrete-time linear systems excited by both state and control multiplicative noise. 

 \textit{Main Contributions:}
First, we present a formulation of the CS problem as a nonlinear program (NLP) based on an affine state feedback policy parametrization and subsequently, we show that this NLP can be transformed into an equivalent semi-definite program by applying suitable variable transformations and semi-definite relaxations. 
Second, we show that SDP relaxations, which are tight in the relaxed problem, are loose in the exact CS problem. 
In view of these results, we propose a two-step procedure to solve the exact CS problem which is based on the solution to the relaxed problem. Third, we provide an instance of the exact CS problem in which the semi-definite relaxations used in the second step of the solution procedure are loose.
Then, we show that the semi-definite relaxations in the second step of the solution procedure are tight if there is no control multiplicative noise acting on the system.

 \textit{Organization of the Paper:} In Section \ref{s:problem-formulation}, we provide the precise problem formulation. 
We present the main theoretical contributions in Section \ref{s:main-results} and we provide solution methods for the relaxed CS problem and the exact CS in Subsections \ref{ss:relaxedCS} and \ref{ss:exactCS}, respectively. 
The results from numerical experiments are presented and discussed in Section \ref{s:numerical-experiments}. 
Finally, we finalize the paper in Section \ref{s:conclusion} with concluding remarks and possible future research directions.

\section{Problem Formulation}\label{s:problem-formulation}
\subsection{Notation}\label{ss:notation}
The space of $n$-dimensional real vectors is denoted as $\R{n}$ and the space of $n\times m$ matrices as $\R{n \times m}$. 
The set of positive integers is denoted as $\Zplus$.
The cone of $n\times n$ positive semi-definite and positive definite matrices are denoted by $\S{n}^{+}$ and $\S{n}^{++}$, respectively. 
$\mathbf{0}$ denotes the zero matrix (or vector) with the appropriate dimension. 
We use $I\sub{n}$ to denote the $n \times n$ identity matrix.
For $A, B \in \S{n}$, $A \succ B$ ($A \succeq B$) means $A - B \in \S{n}^{++}$ ($A - B \in\S{n}^{+}$).
We use $\tr{\cdot}$ to denote the trace operator. 
$\bdiag{A_1, A_2, \dots, A_{N}}$ denotes the block diagonal matrix whose diagonal blocks are the matrices $A_1, A_2, \dots, A_{N}$. 
The expectation and the covariance of a random variable $x$ are denoted as $\E{x}$ and $\Cov{x}$, respectively.
\subsection{Problem Setup and Formulation}\label{ss:problem-setup}
We consider discrete-time linear systems of the form:
\begin{align}\label{eq:system-dynamics}
    &x_{k+1} = \left(A_k + \sum_{\ell=1}^{M} \delta_{k,\ell} \Bar{A}_{k,\ell} \right) x_k \nonumber\\
    &\qquad \quad + \left( B_k + \sum_{\ell=1}^{M} \gamma_{k,\ell} \Bar{B}_{k,\ell} \right) u_k + w_k + d_k
\end{align}
where $x\sub{k} \in \R{n}$, $u\sub{k}\in \R{m}$ are the state and the input processes, respectively.  
We assume that $\E{x\sub{0}} = \mu\sub{0} \in \R{n}$ and $\Cov{x\sub{0}} = \Sigma\sub{0} \in \mathbb{S}\sub{n}\spr{++}$ are given.
$d\sub{k} \in \R{n}$ is known for all $k \in \{ 0, \dots, N-1 \}$.
The state and control multiplicative noise processes are represented by i.i.d. random variables $\delta\sub{k,\ell}, \gamma\sub{k, \ell}$ where $\E{ \delta\sub{k,\ell}} = \E{\gamma\sub{k,\ell} } = 0$ and $\Cov{\delta\sub{k,\ell}} = \Cov{\gamma\sub{k,\ell}} = 1$. 
Note that this representation of the multiplicative noise process is not restrictive i.e. any random matrix $S \in \R{n \times m}$ whose entries have finite second moments can be represented in this form as shown in \cite{p:summers2020sysidmultnoise}.
The additive noise $\{ w\sub{k} \}\sub{k=0}\spr{N-1}$ is also an i.i.d. random process with $\E{w\sub{k}} = \bm{0}$ and $\Cov{w\sub{k}} = W\sub{k} \in \mathbb{S}\sub{n}\spr{+}$.

\begin{remark}
Note that the only assumption that we make on the distributions of the initial state $x\sub{0}$ and the noise processes $w\sub{k}, \delta\sub{k,\ell}, \gamma\sub{k,\ell}$ is that their first two moments are known.
Thus, the distribution of the initial state $x\sub{0}$, the random variables corresponding to multiplicative noise $\delta\sub{k,\ell}, \gamma\sub{k, \ell}$ and additive noise $w\sub{k}$ can have any distribution with given covariance values (they are not necessarily Gaussian).
\end{remark}

A state feedback control policy for the system in \eqref{eq:system-dynamics} is a sequence $\pi = \{ \pi\sub{k}\}\sub{k=0}\spr{N-1}$ where each $\pi\sub{k} : \R{n} \rightarrow \R{m}$ is a function from the state $x\sub{k}$ to control input $u\sub{k}$. 
We denote the set of admissible policies by $\Pi$.

Throughout the paper, we consider a performance measure with a standard quadratic running cost:
\begin{equation}\label{eq:cost-function}
    J(U\sub{0:N-1}, X\sub{0:N}) := \sum_{k=0}^{N-1} u_k\transpose R_k u_k +  x_k\transpose Q_k x_k
\end{equation}
where $U\sub{0:N-1}= \{u\sub{0},\dots, u\sub{N-1} \}$ (input process) and $X\sub{0:N} = \{x\sub{0}, \dots, x\sub{N} \}$ (state process). Now, we can formally state the main problems of interest as follows:

\begin{problem}[Exact Covariance Steering Problem]\label{problem:exact-problem-definition}
Let $N \in \Zplus$,
$\{ A\sub{k}, B\sub{k}, d\sub{k}, W\sub{k}, \{ \Bar{A}\sub{k,\ell} \}\sub{\ell=1}\spr{M}, \{ \Bar{B}\sub{k, \ell} \}\sub{\ell=1}\spr{M} \}\sub{k=0}\spr{N-1}$,
$\mu_0, \mu_\mathrm{d} \in \R{n}, \Sigma_0, \Sigma_\mathrm{d} \in \S{n}^{++}$, 
and 
$\{R\sub{k}, Q\sub{k}\}\sub{k=0}\spr{N-1}$ where $R\sub{k}, Q\sub{k} \in \mathbb{S}\sub{m}\spr{++}, W\sub{k}\in \mathbb{S}\sub{n}\spr{+}$ be given.
Then, find an admissible control policy $\pi\spr{\star} \in \Pi $ that solves the following stochastic optimal control problem:
\begin{subequations}\label{eq:exact-problem}
\begin{align}
    \min_{\pi \in \Pi} &~~ \E{ J(U\sub{0:N-1}, X\sub{0:N}) } \label{eq:first-problem-objective}\\
    \mathrm{s.t.} & ~~ \eqref{eq:system-dynamics} \\
    & ~~ \E{x_N} = \mu_\mathrm{d} \label{eq:first-problem-terminal-mean-constr}\\
    & ~~ \Cov{x_N} = \Sigma_\mathrm{d} \label{eq:first-problem-terminal-cov-constr}\\
    & ~~ u_k = \pi\sub{k}(x\sub{k}) \label{eq:first-problem-policy-constr}
\end{align}
\end{subequations}
\end{problem}
Many practical applications of stochastic optimal control problems require the terminal covariance of the state to be 
upper bounded by some acceptable covariance matrix in the L\"{o}wner partial order sense.
Therefore, we also consider the `relaxed' variation of the exact CS problem in which the terminal covariance constraint in \eqref{eq:first-problem-terminal-cov-constr} is relaxed to the LMI constraint in \eqref{eq:relaxed-problem-terminal-cov-constr}.

\begin{problem}[Relaxed Covariance Steering Problem]\label{problem:relaxed-problem-definition}
Let $N \in \Zplus$,
$\{ A\sub{k}, B\sub{k}, d\sub{k}, W\sub{k}, \{ \Bar{A}\sub{k,\ell} \}\sub{\ell=1}\spr{M}, \{ \Bar{B}\sub{k, \ell} \}\sub{\ell=1}\spr{M} \}\sub{k=0}\spr{N-1}$,
$\mu_0, \mu_\mathrm{d} \in \R{n}, \Sigma_0, \Sigma_\mathrm{d} \in \S{n}^{++}$, 
and 
$\{R\sub{k}, Q\sub{k}\}\sub{k=0}\spr{N-1}$ where $R\sub{k} \in \mathbb{S}\sub{m}\spr{++}, W\sub{k}, Q\sub{k} \in \mathbb{S}\sub{n}\spr{+}$ be given.
Then, find an admissible control policy $\pi\spr{\star} \in \Pi $ that solves the following stochastic optimal control problem:
\begin{subequations}\label{eq:relaxed-problem}
\begin{align}
    \min_{\pi \in \Pi} &~~ \E{J(U\sub{0:N-1}, X\sub{0:N})} \\
    \mathrm{s.t.} & ~~ \eqref{eq:system-dynamics}, \eqref{eq:first-problem-terminal-mean-constr},  \eqref{eq:first-problem-policy-constr} \\
    & ~~ \Cov{x_N} \preceq \Sigma_\mathrm{d} \label{eq:relaxed-problem-terminal-cov-constr}
\end{align}
\end{subequations}
\end{problem}

\begin{remark}
Apart from having practical importance, our proposed solution procedure for Problem \ref{problem:exact-problem-definition} (exact CS) requires the optimal policy parameters obtained by solving Problem \ref{problem:relaxed-problem-definition} as explained in Section \ref{s:main-results}.
\end{remark}

\section{Main Results}\label{s:main-results}
Since the proposed solution method for Problem \ref{problem:exact-problem-definition} requires the solution obtained by solving Problem \ref{problem:relaxed-problem-definition}, we first present our results on the relaxed CS problem. 
Both Problem \ref{problem:exact-problem-definition} and Problem \ref{problem:relaxed-problem-definition} are stochastic optimal control problems over infinite dimensional policy spaces which make them computationally intractable for most cases. 
However, it has been shown that the optimal policy for the CS problems is in the form of an affine state feedback for both continuous-time \cite{p:chen2015covariance1} and discrete-time \cite{p:balci2021convexity-wasserstein, p:goldshtein2017covariance, p:ito2022maxentropyCS}. 
Thus, we restrict the set of policies that we optimize over to the set of affine state feedback policies which is denoted as $\Pi\spr{sf}$. 
In particular, a policy $\pi = \{\pi\sub{k}\}\sub{k=0}\spr{N-1} \in \Pi\spr{sf}$ is given as 
\begin{align}\label{eq:state-feedback-control-policy}
    \pi_k (x\sub{k}) = \Bar{u}_k + K_k (x_k - \mu_k),
\end{align}
for every $\pi\sub{k}:\R{n} \rightarrow \R{m}$ where $\mu\sub{k} = \E{x\sub{k}}$.
With this formulation, the policy space $\Pi\spr{sf}$ is parametrized by a finite number of decision variables which are $\{\Bar{u}\sub{k}, K\sub{k} \}\sub{k=0}\spr{N-1}$ where $\ubar\sub{k} \in \R{M}, K\sub{k}\in \R{m\times n}$.

Under the policy parameterization defined in \eqref{eq:state-feedback-control-policy}, the mean and the covariance dynamics of the state process $x\sub{k}$ obey the following recursive equations:
\begin{align}
    \mu_{k+1} &= A_k \mu_k + B_k \Bar{u}_k + d_k \label{eq:mean-system-dynamics} \\ 
    \Sigma_{k+1} & = (A_k + B_k K_k) \Sigma_k (A_k + B_k K_k)\transpose + W_k \nonumber \\
    & ~~~ + \sum_{\ell=0}^{M} \Bar{B}_{k,\ell} (K_k \Sigma_k K_k\transpose + \ubar\sub{k} \ubar\sub{k}\transpose) \Bar{B}_{k,\ell}\transpose \nonumber \\
    & ~~~ + \sum_{\ell=0}^{M} \Bar{A}_{k,\ell} (\Sigma_k + \mu_k \mu_k\transpose) \Bar{A}_{k,\ell}\transpose \label{eq:covariance-system-dynamics}
\end{align}
where $\Sigma\sub{k} = \Cov{x\sub{k}}$ and we used the fact that random variables $\delta\sub{k, \ell}$ and $\gamma\sub{k,\ell}$ are i.i.d. random processes. 
Besides the system dynamics, we need to represent the objective  function $\E{J(U\sub{0:N-1}, X\sub{0:N})}$ in terms of the policy parameters $\{ \Bar{u}\sub{k}, K\sub{k}\}\sub{k=0}\spr{N-1}$ where $J(U\sub{0:N-1}, X\sub{0:N})$ is defined in \eqref{eq:cost-function} to formulate both Problem \ref{problem:exact-problem-definition} and Problem \ref{problem:relaxed-problem-definition} as finite dimensional nonlinear programs.
To this aim, we use the following identities:
\begin{subequations}\label{eq:quadratic-cost-transformations}
\begin{align}
    \E{u_k\transpose R_k u_k} & = \tr{R_k (\Cov{u_k} + \Bar{u}_k \Bar{u}_k\transpose) } \nonumber \\
    & = \tr{R_k \Bar{u}_k \Bar{u}_k\transpose } + \tr{R_k K_k \Sigma_k K_k\transpose} \\
    \E{x_k\transpose Q_k x_k} & = \tr{Q_k (\Cov{x_k} + \mu_k \mu_k\transpose)} \nonumber \\
    & = \tr{Q_k \mu_k \mu_k\transpose} + \tr{Q_k \Sigma_k}
\end{align}
\end{subequations}
whose derivation is based on the linearity of the expectation operator $\E{\cdot}$ and the cyclic permutation property of the trace operator $\tr{\cdot}$.
By the summation of the equalities in  \eqref{eq:quadratic-cost-transformations} over all $k \in \{0, \dots, N-1 \}$, we observe that
\begin{align}\label{eq:calJdefinition}
    \E{J(U\sub{0:N-1}, X\sub{0:N})} & = \sum\sub{k=0}\spr{N-1} \tr{R\sub{k} (\Bar{u}\sub{k} \Bar{u}\sub{k}\transpose + K\sub{k} \Sigma\sub{k} K\sub{k}\transpose ) } \nonumber \\
& \qquad \quad + \tr{Q\sub{k} (\mu\sub{k} \mu\sub{k}\transpose + \Sigma\sub{k})} \nonumber \\
& =: \cJ (\{ \Bar{u}\sub{k}, K\sub{k}, \mu\sub{k}, \Sigma\sub{k} \}\sub{k=0}\spr{N-1})
\end{align}

Now that we have written both the dynamics of the mean $\mu\sub{k}$ and the covariance $\Sigma\sub{k}$ of the state and the objective function $\E{J(U\sub{0:N-1}, X\sub{0:N})}$ in term of policy parameters $\{ \Bar{u}\sub{k}, K\sub{k} \}\sub{k=0}\spr{N-1}$, we are ready to formulate Problem \ref{problem:exact-problem-definition} and \ref{problem:relaxed-problem-definition} as finite dimensional optimization problems.


\subsection{Relaxed Covariance Steering}\label{ss:relaxedCS}

A finite dimensional optimization problem over the variables $\{ \mu\sub{k}, \Sigma\sub{k}, \Bar{u}\sub{k}, K\sub{k}\}$ can be written as follows:
\begin{subequations}\label{eq:finite-dim-NLP-relaxed}
\begin{align}
    \min_{\substack{\Bar{u}_k, K_k \\ \mu\sub{k}, \Sigma\sub{k}}} & ~~ \cJ ( \{ \Bar{u}\sub{k}, K\sub{k}, \mu\sub{k}, \Sigma\sub{k} \}\sub{k=0}\spr{N-1} ) \label{eq:finite-dim-NLP-relaxed-objective} \\
    \text{s.t.} & ~~ \eqref{eq:mean-system-dynamics}, \eqref{eq:covariance-system-dynamics} \nonumber \\
    & ~~ \mu\sub{N} = \mu_\mathrm{d} \\
    & ~~ \Sigma_\mathrm{d} \succeq \Sigma\sub{N}
\end{align}
\end{subequations}
The optimization problem in \eqref{eq:finite-dim-NLP-relaxed} is a general nonlinear program (NLP) which poses computational challenges since there is no guarantee of convergence to the globally optimal solution. 
The hardness of the problem in \eqref{eq:finite-dim-NLP-relaxed} comes from the bilinear terms $ K\sub{k} \Sigma\sub{k}$ and  $K\sub{k}\Sigma\sub{k}K\sub{k}\transpose$ that appear in the objective function $\cJ(\{ \Bar{u}\sub{k}, K\sub{k}, \mu\sub{k}, \Sigma\sub{k} \}\sub{k=0}\spr{N-1})$ and the state covariance dynamics given in \eqref{eq:covariance-system-dynamics} since the mean dynamics in \eqref{eq:mean-system-dynamics} is affine and the other terms in the objective function in \eqref{eq:finite-dim-NLP-relaxed-objective} are either affine or convex quadratic function of the decision variables. 

To isolate the nonlinearities in the optimization problem in \eqref{eq:finite-dim-NLP-relaxed}, we introduce the following decision variables: 
\begin{equation}\label{eq:new-variables-definition}
\begin{aligned}
    L\sub{k} &= K\sub{k} \Sigma\sub{k},~& \mathbf{M}\sub{k} &= L\sub{k}\Sigma\sub{k}\spr{-1}L\sub{k}\transpose, \\
    \mathbf{X}\sub{k} &= \mu\sub{k} \mu\sub{k}\transpose, ~~ &\mathbf{U}\sub{k} &= \Bar{u}\sub{k}\Bar{u}\sub{k}\transpose.
\end{aligned}
\end{equation}
The objective function $\cJ (\cdot)$ in \eqref{eq:calJdefinition} can be equivalently represented in terms of the new decision variables which are defined in \eqref{eq:new-variables-definition} as follows:
\begin{equation}
\begin{aligned}
    &\Hat{\cJ}( \{ \mathbf{M}\sub{k}, \mathbf{U}\sub{k}, \Sigma\sub{k}, \mathbf{X}\sub{k} \}\sub{k=0}\spr{N-1}) := \\
    &\qquad \quad \sum_{k=0}^{N-1} \tr{R_k (\mathbf{U}\sub{k} + \mathbf{M}\sub{k})} + \tr{Q\sub{k} (\Sigma\sub{k} + \mathbf{X}\sub{k}) }.
\end{aligned}
\end{equation}

By using the new decision variables, we can formulate a new optimization problem that is equivalent to \eqref{eq:finite-dim-NLP-relaxed} as follows:
\begin{subequations}\label{eq:finite-dim-nlp-relaxed-new}
\begin{align}
    \min_{\substack{\Bar{u}\sub{k}, L\sub{k}, \Sigma\sub{k}, \\ \mathbf{M}_k, \mathbf{X}_k, \mathbf{U}_k}} &~~ \Hat{\cJ}(\{ \mathbf{M}\sub{k}, \mathbf{U}\sub{k}, \Sigma\sub{k}, \mathbf{X}\sub{k} \}\sub{k=0}\spr{N-1})   \label{eq:finite-dim-nlp-relaxed-new-objective}\\
    \text{s.t.} & ~~ \mu\sub{k+1} = A\sub{k} \mu\sub{k} + B\sub{k} \Bar{u}\sub{k} + d\sub{k}, \label{eq:finite-dim-nlp-relaxed-new-mean-dyn-constr}\\ 
    & ~~ \Sigma_{k+1} = A_k \Sigma_k A_k\transpose + A_k L_k\transpose B_k\transpose \nonumber \\
    & \qquad  + B_k L_k A_k\transpose + B_k \mathbf{M}_k B_k\transpose + W_k \nonumber \\
    & \qquad + \sum_{\ell \eq 0}^{M} \Bar{B}_{k,\ell} (\mathbf{M}_k + \mathbf{U}_k) \Bar{B}_{k,\ell}\transpose \nonumber \\
    & \qquad + \sum_{\ell=0}^{M} \Bar{A}_{k,\ell} (\Sigma_k + \mathbf{X}_k) \Bar{A}_{k,\ell}\transpose,  \label{eq:finite-dim-nlp-relaxed-new-cov-dyn-constr}\\
    & ~~ \mathbf{M}_k = L\sub{k} \inv{\Sigma}\sub{k} L\sub{k}\transpose, \label{eq:finite-dim-nlp-relaxed-new-M-constr}\\
    & ~~  \mathbf{X}_k \, = \mu_k \mu_k\transpose, \label{eq:finite-dim-nlp-relaxed-new-X-constr}\\
    & ~~ \mathbf{U}_k \, = \Bar{u}_k \Bar{u}_k\transpose \label{eq:finite-dim-nlp-relaxed-new-U-constr}\\
    & ~~ \mu\sub{N} \, = \mu\sub{\mathrm{d}} \label{eq:finite-dim-nlp-relaxed-new-terminal-constr},  ~~~ \Sigma\sub{\mathrm{d}} \succeq \Sigma\sub{N}
\end{align}
\end{subequations}
where we replaced the bilinear terms $K\sub{k} \Sigma\sub{k}$ in the 
recursive equation for the propagation of the state covariance \eqref{eq:covariance-system-dynamics}
with $L\sub{k}$.
The term $K\sub{k} \Sigma\sub{k}\spr{-1} K\sub{k}\transpose$ is rewritten as $K\sub{k} \Sigma\sub{k} \Sigma\sub{k}\spr{-1} \Sigma\sub{k} K\sub{k}$ then turned into $L\sub{k} \Sigma\sub{k}\spr{-1} L\sub{k}\transpose$ which is then replaced with $\mathbf{M}\sub{k}$. 
The terms $\mu\sub{k}\mu\sub{k}\transpose$, $\Bar{u}\sub{k} \Bar{u}\sub{k}\transpose$ are replaced with $\mathbf{X}\sub{k}$, $\mathbf{U}\sub{k}$, respectively.
Note that the constraints that include the decision variables denoted with subscript $k$ are imposed for all $k \in \{ 0, \dots, N-1\}$ in the rest of the optimization problems defined throughout the paper. 
Finally, to keep the equivalence of the problems in \eqref{eq:finite-dim-NLP-relaxed} and \eqref{eq:finite-dim-nlp-relaxed-new}, we add the nonlinear equalities in \eqref{eq:new-variables-definition} as constraints in \eqref{eq:finite-dim-nlp-relaxed-new-M-constr}, \eqref{eq:finite-dim-nlp-relaxed-new-X-constr} and \eqref{eq:finite-dim-nlp-relaxed-new-U-constr}. 


After introducing the new decision variables, the problem in \eqref{eq:finite-dim-NLP-relaxed} takes the form in \eqref{eq:finite-dim-nlp-relaxed-new} where the objective function now is expressed as affine functions of the decision variables $\Sigma\sub{k}, \mathbf{X}\sub{k}, \mathbf{U}\sub{k}, \mathbf{M}\sub{k}$.
Furthermore, the covariance dynamics constraint in \eqref{eq:finite-dim-nlp-relaxed-new-cov-dyn-constr} is now represented as an affine constraint in decision variables. 
Unfortunately, the problem defined in \eqref{eq:finite-dim-nlp-relaxed-new} is still not a convex optimization problem due to the nonlinear equality constraints \eqref{eq:finite-dim-nlp-relaxed-new-M-constr}- \eqref{eq:finite-dim-nlp-relaxed-new-U-constr}. 
To convefixy the problem, we relax the nonlinear equality constraints \eqref{eq:finite-dim-nlp-relaxed-new-M-constr}-\eqref{eq:finite-dim-nlp-relaxed-new-U-constr} as follows:
\begin{equation}\label{eq:relaxed-nonlinear-equalities}
    \mathbf{M}\sub{k} \succeq L\sub{k} \Sigma\sub{k}\spr{-1} L\sub{k}, ~~~ \mathbf{X}\sub{k} \succeq \mu\sub{k}\mu\sub{k}\transpose, ~~~ \mathbf{U}\sub{k} \succeq \Bar{u}\sub{k}\Bar{u}\sub{k}\transpose.
\end{equation}
In light of Schur's complement lemma \cite{b:zhang2006schur},
the relaxed nonlinear SDP constraints in \eqref{eq:relaxed-nonlinear-equalities} can be transformed into LMI constraints in \eqref{eq:relaxed-prob-relaxed-M-constr}-\eqref{eq:relaxed-prob-relaxed-U-constr}.
The resulting optimization problem after the SDP relaxations is given as follows:

\begin{subequations}\label{eq:relaxed-finite-dim-opt-problem}
\begin{align}
    \min_{\substack{\Bar{u}\sub{k}, \mu\sub{k},  L\sub{k}, \Sigma\sub{k}, \\ \mathbf{M}\sub{k}, \mathbf{X}_k, \mathbf{U}_k}} & ~~ \Hat{\cJ} (\{ \mathbf{M}\sub{k}, \mathbf{U}\sub{k}, \Sigma\sub{k}, \mathbf{X}\sub{k} \}\sub{k=0}\spr{N-1})  \nonumber \\
    \text{s.t.} & ~~ \eqref{eq:finite-dim-nlp-relaxed-new-mean-dyn-constr}, \eqref{eq:finite-dim-nlp-relaxed-new-cov-dyn-constr}, \eqref{eq:finite-dim-nlp-relaxed-new-terminal-constr}, \nonumber \\
    & \begin{bmatrix} \mathbf{M}\sub{k} & L\sub{k} \\
    L\sub{k} & \Sigma\sub{k} \end{bmatrix} \succeq \bm{0}, \label{eq:relaxed-prob-relaxed-M-constr} \\
    &  \begin{bmatrix} \mathbf{X}\sub{k} & \mu\sub{k} \\
    \mu\transpose\sub{k} & 1 \end{bmatrix} \succeq \bm{0}, \\
    & \begin{bmatrix} \mathbf{U}\sub{k} & \Bar{u}\sub{k} \\
    \Bar{u}\sub{k}\transpose & 1 \end{bmatrix} \succeq \bm{0}, \\
    & ~ \Sigma\sub{\mathrm{d}} \succeq \Sigma\sub{N}, \label{eq:relaxed-prob-relaxed-U-constr}
\end{align}
\end{subequations}

To be able to recover the optimal state feedback policy parameters $\{ \Bar{u}\sub{k}, K\sub{k} \}\sub{k=0}\spr{N-1}$ from the solution of the SDP in \eqref{eq:relaxed-finite-dim-opt-problem} which is denoted as $ ( \{ \Bar{u}\sub{k}\spr{\star}, L\sub{k}\spr{\star}, \Sigma\sub{k}\spr{\star}, \mathbf{M}\sub{k}\spr{\star}, \mathbf{X}\sub{k}\spr{\star}, \mathbf{U}\sub{k}\spr{\star}  \}_{k=0}^{N-1} ) $, we need the optimal parameters to satisfy the relaxed (non-strict) inequality constraints in \eqref{eq:relaxed-nonlinear-equalities} with equality.
Next, we show that the optimal parameters of problem in \eqref{eq:relaxed-finite-dim-opt-problem} satisfy the nonlinear equality constraints \eqref{eq:finite-dim-nlp-relaxed-new-M-constr}-\eqref{eq:finite-dim-nlp-relaxed-new-U-constr}.
\begin{proposition}\label{prop:lossless-sdp-relaxed}
Let $ \{ \Bar{u}\sub{k}\spr{\star}, L\sub{k}\spr{\star}, \Sigma\sub{k}\spr{\star}, \mathbf{M}\sub{k}\spr{\star}, \mathbf{X}\sub{k}\spr{\star}, \mathbf{U}\sub{k}\spr{\star}  \}_{k=0}^{N-1}$ be the optimal solution of Problem in \eqref{eq:relaxed-finite-dim-opt-problem}. Then, it satisfies the equalities in \eqref{eq:finite-dim-nlp-relaxed-new-M-constr}, \eqref{eq:finite-dim-nlp-relaxed-new-X-constr}, \eqref{eq:finite-dim-nlp-relaxed-new-U-constr}. 
Therefore, it is an optimal solution to Problem \ref{problem:relaxed-problem-definition}.
\end{proposition}
\begin{proof}
Suppose for the sake of contradiction that the parameters corresponding to the optimal solution satisfy $\mathbf{M}\sub{k} - L\sub{k} \Sigma\sub{k}\spr{-1} L\sub{k}\transpose = \mathbf{N}\spr{m} \neq \bm{0}$, $\mathbf{X}\sub{k} - \mu\sub{k}\mu\sub{k}\transpose = \mathbf{N}\spr{x} \neq \bm{0}$, $\mathbf{U}\sub{k} - \Bar{u}\sub{k}\Bar{u}\sub{k}\transpose = \mathbf{N}\spr{u} \neq \bm{0}$ for some $k \in \{0, \dots, N-1 \}$. 
Now, let's define 
$\mathbf{M}\sub{k}' = L\sub{k}\Sigma\sub{k}\spr{-1} L\sub{k}\transpose$, 
$\mathbf{U}\sub{k}' = \Bar{u}\sub{k}\Bar{u}\sub{k}$ 
and 
$\bfX'\sub{k} = \mu\sub{k} \mu\sub{k}\transpose$.
It follows readily that
$\mathbf{M}\sub{k} \succeq \mathbf{M}\sub{k}'$, 
$\mathbf{U}\sub{k} \succeq \mathbf{U}\sub{k}'$
and
$\bfX\sub{k} \succeq \bfX'\sub{k}$.
Since $R\sub{k} \succ 0$, $Q\sub{k} \succ \bm{0}$; $\tr{R\sub{k} (\mathbf{M}\sub{k} + \mathbf{U}\sub{k})} > \tr{R\sub{k} (\mathbf{M}\sub{k}' + \mathbf{U}\sub{k}')}$. 
Thus, the value of the objective function is strictly lower with $ \mathbf{M}\sub{k}', \mathbf{U}\sub{k}', \bfX'\sub{k}$. 
Furthermore, let $\Sigma\sub{t}'$ be the value of the state covariance under $\mathbf{M}\sub{k}', \mathbf{U}\sub{k}', \bfX'\sub{k}$ for all $t \geq k+1$. 
Then, we have that $\Sigma\sub{t} \succeq \Sigma\sub{t}'$ for all $t \geq k+1$. 
Now, replace $L\sub{t}$ with $L\sub{t} (\Sigma\sub{t}')\spr{-1} \Sigma\sub{t}$ to satisfy feasibility of constraint \eqref{eq:relaxed-prob-relaxed-M-constr} for all $t \geq k+1$. 
Since $\bfM\sub{t} \succeq L\sub{t} \Sigma\sub{t}\spr{-1} L\sub{t}\transpose = L\sub{t}' (\Sigma\sub{t}')\spr{-1} \Sigma\sub{t} \Sigma\sub{t}\spr{-1} \Sigma\sub{t} (\Sigma\sub{t}')\spr{-1} L\sub{t}\spr{\prime\mathrm{T}}$. 
Since $\Sigma\sub{t} \succeq \Sigma\sub{t}'$ implies that $(\Sigma\sub{t}')\spr{-1} \Sigma\sub{t}  (\Sigma\sub{t}')\spr{-1} \succeq (\Sigma\sub{t}')\spr{-1}$; we have $\bfM\sub{t} \succeq L\sub{t}' (\Sigma\sub{t}')\spr{-1} L\sub{t}'$ thus the constraint \eqref{eq:relaxed-prob-relaxed-M-constr} is satisfied. 
Combining both results, we conclude that if the inequalities in \eqref{eq:relaxed-nonlinear-equalities} are not strict, one can pick new values for $\bfM\sub{k}, L\sub{k}$ which decrease the value of the objective function without violating the constraints which contradicts the optimality assumption of $\bfM\sub{k}, L\sub{k}$.
This completes the proof.
\end{proof}



\subsection{Exact Covariance Steering}\label{ss:exactCS}
For the covariance upper bound constraint in \eqref{eq:relaxed-problem-terminal-cov-constr}, CS problem \ref{problem:exact-problem-definition} can be relaxed into problem in \eqref{eq:relaxed-finite-dim-opt-problem} without changing the nature of the problem according to Proposition \ref{prop:lossless-sdp-relaxed}. 
However, the SDP relaxations in \eqref{eq:relaxed-nonlinear-equalities} for the constraints in \eqref{eq:new-variables-definition} for $\mathbf{X}\sub{k}, \mathbf{U}\sub{k}$ do not hold with equality in the exact covariance steering problem (Problem \ref{problem:exact-problem-definition}). 

In our numerical experiments,
we observed that the loose constraints were the ones with $\mathbf{X}\sub{k}, \mathbf{U}\sub{k}$ in the optimal solution. 
Furthermore, one can show that if the feed-forward control inputs ($\{ \Bar{u}\sub{k} \}\sub{k=0}\spr{N-1}$) are fixed, then the state mean $\mu\sub{k}$ is also fixed through \eqref{eq:mean-system-dynamics} thus $\mathbf{X}\sub{k}, \mathbf{U}\sub{k}$ can be set to their respected values for fixed $\Bar{u}\sub{k}$. 

Now, suppose that the relaxed CS problem is feasible and let $\{\Bar{u}\sub{k}\spr{\star}, K\sub{k}\spr{\star} \}\sub{k=0}\spr{N-1}$, $\{\mu\sub{k}\spr{\star}, \Sigma\sub{k}\spr{\star} \}\sub{k=0}\spr{N}$ be the policy parameters and the state statistics that is found by solving problem in \eqref{eq:finite-dim-nlp-relaxed-new}, respectively.
Then, we have that $\mu\sub{N} = \mu\sub{\mathrm{d}}$, and the terminal covariance constraint $\Sigma\sub{\mathrm{d}} \succeq \Sigma\sub{N} $ is also satisfied.

After $\ubar\sub{k}, \mu\sub{k}$ are fixed based on the values obtained from solving \eqref{eq:relaxed-finite-dim-opt-problem}, the decision variables $\{\ubar\sub{k}, \mu\sub{k}, \bfX\sub{k}, \bfM\sub{k} \}\sub{k=0}\spr{N-1} $ become problem parameters for the exact CS problem.
Thus, we formulate another optimization problem with $L\sub{k}, \Sigma\sub{k}, \bfM\sub{k}$ as the decision variables as follows:
\begin{subequations}\label{eq:covariance-relaxed-problem}
\begin{align}
    \min\sub{L\sub{k}, \Sigma\sub{k}, \mathbf{M}\sub{k}} & ~~ \Tilde{\cJ}(\{ \bfM\sub{k}, \Sigma\sub{k}\}\sub{k=0}\spr{N-1}) \\
    \text{s.t.} & ~~ \Sigma\sub{k+1} = A\sub{k} \Sigma\sub{k} A\sub{k}\transpose + A\sub{k} L\sub{k}\transpose B\sub{k}\transpose + B\sub{k}L\sub{k} A\sub{k}\transpose \nonumber \\
    & ~~ + B\sub{k} \mathbf{M}\sub{k} B\sub{k}\transpose + \mathbf{H}\sub{k} + \nonumber \\
    & ~~ + \sum\sub{\ell=1}\spr{M}  \big( \Bar{A}\sub{k,\ell} \Sigma\sub{k} \Bar{A}\sub{k,\ell}\transpose + \Bar{B}\sub{k,\ell} \mathbf{M}\sub{k} \Bar{B}\sub{k,\ell}\transpose ) \\
    & \begin{bmatrix} \mathbf{M}\sub{k} & L\sub{k} \\ L\sub{k}\transpose & \Sigma\sub{k} \end{bmatrix} \succeq \bm{0}  \label{eq:covariance-relaxed-problem-Mk-constr}\\
    & ~~ \Sigma\sub{N} = \Sigma\sub{\mathrm{d}}
\end{align}
\end{subequations}
where $ \Tilde{\cJ} (\{ \bfM\sub{k}, \Sigma\sub{k} \}\sub{k=0}\spr{N-1}) :=  \sum\sub{k=0}\spr{N-1} \tr{R\sub{k} \mathbf{M}\sub{k} + Q\sub{k} \Sigma\sub{k}}$,  $\mathbf{H}\sub{k} = W\sub{k} + \sum\sub{\ell=1}\spr{M} ( \Bar{A}\sub{k,\ell} \mathbf{X}\sub{k}\spr{\star} \Bar{A}\sub{k,\ell}\transpose + \Bar{B}\sub{k,\ell} \mathbf{U}\sub{k,\ell}\spr{\star} \Bar{B}\sub{k,\ell}\transpose )$,
$\mathbf{U}\sub{k}\spr{\star} = \Bar{u}\sub{k}\spr{\star} \Bar{u}\sub{k}\spr{\star \mathrm{T}}$
and $\mathbf{X}\sub{k}\spr{\star} = \mu\sub{k}\spr{\star} \mu\sub{k}\spr{\star\mathrm{T}}$. 

To recover the optimal state feedback control parameters $\{ K\sub{k}\}\sub{k = 0}\spr{N-1}$ from the identity $K\sub{k} =  L\sub{k} \Sigma\sub{k}\spr{-1}$, we need the optimal solution of the problem in \eqref{eq:covariance-relaxed-problem} to satisfy the equality $\mathbf{M}\sub{k} = L\sub{k} \Sigma\sub{k}\spr{-1} L\sub{k}\transpose$, otherwise the recovered policy will not satisfy the terminal covariance constraint. 
Although we observed that the LMI constraint in \eqref{eq:covariance-relaxed-problem-Mk-constr} holds with equality in all of our numerical experiments in Section \ref{s:numerical-experiments}, this may not always be the case. 
The next problem instance is one example of such cases where the LMI constraint in \eqref{eq:covariance-relaxed-problem-Mk-constr} is loose. 
\begin{example}\label{example:loose-inequality}
Let parameters of the example problem instance be given as: 
$ N = 1$, 
$A\sub{0} = \left[ \begin{smallmatrix} 1.04 & -0.22 \\ -0.07 & 1.341 \end{smallmatrix} \right ]$, 
$B\sub{k} = \begin{smat} -0.5 \\ -0.38\end{smat}$, 
$A\sub{0,1} = \begin{smat} -0.16 & -0.2 \\ -0.14 & 0.24 \end{smat}$, 
$ \Bar{B}\sub{0, 1} = \begin{smat} 0.26 \\ -0.16 \end{smat} $, 
$d\sub{k} = \bm{0}$, $W\sub{0} = \bm{0}$, 
$\mu\sub{0} = \mu\sub{\mathrm{d}} = \bm{0}$,
$R\sub{0} = 10.0$,
$Q\sub{0} = 0.1 I\sub{2}$,
$\Sigma\sub{0} = I\sub{2}$, 
$\Sigma\sub{d} = \begin{smat} 1.26 & -0.36 \\ -0.36 & 1.91 \end{smat}$.
The terminal mean constraint in \eqref{eq:first-problem-terminal-mean-constr} dictates that $ A\sub{0} \mu\sub{0} + B\sub{0} \Bar{u}\sub{0} + d\sub{0} = \mu\sub{1} =  \mu\sub{\mathrm{d}} = \bm{0}$. 
Since $\mu\sub{0} = \mu\sub{\mathrm{d}} = d\sub{0} = \bm{0}$ then it follows that $B\sub{0} \Bar{u}\sub{0} = \bm{0}$ which implies that $\ubar\sub{0} = 0$ assuming that $B\sub{0}$ is full-rank.
Now that $\ubar$ is fixed to $0$ the optimal solution for Problem \ref{problem:relaxed-problem-definition} for this given instance can be obtained by solving the SDP in \eqref{eq:covariance-relaxed-problem}.
By solving the aforementioned SDP using MOSEK \cite{mosek}, we obtain the following optimal values for decision variables:
\begin{subequations}
\begin{align}
    \mathbf{M}\sub{0} = 0.149, ~~~~ L\sub{0} = [-0.0181, -0.008] \nonumber
\end{align}
which yields
\begin{align}
    \mathbf{M}\sub{0} -  L\sub{0}\Sigma\sub{0}\spr{-1}L\sub{0}\transpose = 0.148 \neq 0 \nonumber
\end{align}
\end{subequations}
which shows that for the given problem instance, the constraint in \eqref{eq:covariance-relaxed-problem-Mk-constr} is loose. 
\end{example}

Although the solution of the problem instance in Example \ref{example:loose-inequality} does not correspond to an affine state feedback policy since $\mathbf{M}\sub{k} = L\sub{k} \Sigma\sub{k}\spr{-1} L\sub{k}\transpose$ is not satisfied,
the mean and the covariance of the state and the control processes which can be found by solving \eqref{eq:covariance-relaxed-problem} can still be realized by considering randomized affine state feedback policies as in \cite{p:balci2022exactcovariancewasserstein}. 

The set of randomized affine state feedback policies is denoted by $\Pi\spr{rsf}$. 
Every $\pi \in \Pi\spr{rsf}$ is a sequence $\pi = \{ \pi\sub{k}\}\sub{k=0}\spr{N-1}$ where each $\pi\sub{k}$ is given by: \begin{equation}\label{eq:randomized-state-feedback}
    \pi\sub{k} = \Bar{u}\sub{k} + K\sub{k} (x\sub{k} - \mu\sub{k}) + v\sub{k}
\end{equation}
where $v\sub{k} \in \R{m}$ is a random variable with $\E{v\sub{k}} = \bm{0}$, $\Cov{v\sub{k}} = P\sub{k} \in \mathbb{S}\spr{+}\sub{n}$ and each $v\sub{k}$ satisfies that $\E{v\sub{k} x\sub{\ell}\transpose} = \bm{0}$ for all $\ell \leq k$, $\E{v\sub{k} \delta\sub{n, \ell }} = \E{v\sub{k} \gamma\sub{n, \ell}} = \bm{0}$, $\E{v\sub{k} w\sub{\ell}\transpose} = \bm{0}$ for all $n , \ell$.
Thus, the randomized affine state feedback policies are parametrized by the decision variables $\{ \ubar\sub{k}, K\sub{k}, P\sub{k} \}\sub{k=0}\spr{N-1}$. 
Now, setting the parameters of the randomized policy to $K\sub{k} = L\sub{k}\Sigma\sub{k}\spr{-1}$ and $ P\sub{k} = \bfM\sub{k} - L\sub{k}\Sigma\sub{k}\spr{-1} L\sub{k}\transpose$, the randomized affine state feedback policy induces a state process and a control process whose first and second moments are equal to the ones found by solving \eqref{eq:covariance-relaxed-problem}.

Despite the fact that deterministic affine state feedback policies are sufficient for CS problems for systems excited by additive noise \cite{p:chen2015covariance1, p:balci2022exactcovariancewasserstein, p:goldshtein2017covariance},
Example \ref{example:loose-inequality} shows that the optimal policy for exact CS problem \ref{problem:exact-problem-definition} may require randomized policies for systems excited by multiplicative noise.

If we consider the special case of Problem \ref{problem:exact-problem-definition} where the multiplicative noise only acts through the state 
which means that $\Bar{B}\sub{k, \ell} = \bm{0}$ for all $k , \ell$. 
Then, we can show that the LMI constraint in \eqref{eq:covariance-relaxed-problem-Mk-constr} holds with equality. 
The next proposition formally states that claim. 
\begin{proposition}\label{prop:proposition2}
Assuming that the problem in \eqref{eq:covariance-relaxed-problem} is feasible, $ \Bar{B}\sub{k, \ell} = \bm{0}$ and $A\sub{k}\spr{-1}$ exists for all $k , \ell$, then the optimal values of decision variables $\{ L\sub{k}\spr{\star}, \Sigma\sub{k}\spr{\star}, \mathbf{M}\sub{k}\spr{\star} \}\sub{k=0}\spr{N-1}$ satisfy $ \mathbf{M}\sub{k}\spr{\star} = L\spr{\star}\sub{k} \Sigma\sub{k}\spr{\star -1} L\sub{k}\spr{\star \mathrm{T}}$.
\end{proposition}
\begin{proof}
Let $\bfM\spr{\star}\sub{t} - L\spr{\star}\sub{t} \Sigma\sub{t}\spr{\star - 1} L\sub{t}\transpose \neq \bm{0}$ for some $t \in \{0, \dots, N-1\}$ for the sake of contradiction.
Let us consider the following SDP:
\begin{subequations}\label{eq:prop2helpersdp}
\begin{align}
    \min\sub{\substack{L \in \R{m\times n} \\\bfM \in \mathbb{S}\sub{m}\spr{+}}} & ~~ \tr{R\sub{t} \bfM} \label{eq:helpersdp-objective}\\
    \text{s.t.} & ~~ \Sigma\sub{t+1}\spr{\star} =  A\sub{t} L\transpose B\sub{t}\transpose + B\sub{t} L A\sub{t}\transpose \nonumber \\ 
    & \qquad \qquad + B\sub{t} \mathbf{M} B\sub{t}\transpose + \mathbf{H}\sub{t}   \label{eq:helpersdp-constr1}\\
    & ~~ \begin{bmatrix} \bfM & L \\ L\transpose & \Sigma\sub{t}\spr{\star} \end{bmatrix} \succeq \bm{0} \label{eq:helpersdp-constr2}
\end{align}
\end{subequations}
where $\mathbf{H}\sub{t} = W\sub{t} + A\sub{t} \Sigma\spr{\star}\sub{t}  A\sub{t}\transpose + \sum\sub{\ell=1}\spr{M} \Bar{A}\sub{t, \ell} (\Sigma\spr{\star}\sub{t} + \bfX\sub{t}) \Bar{A}\sub{t, \ell}\transpose$. 
The SDP in \eqref{eq:prop2helpersdp} represents the covariance evolution from time step $t$ to $t+1$ but covariance values are fixed. 
So, the objective is to find policy parameters ${\bfM, L}$ to steer the covariance from $\Sigma\sub{t}\spr{\star}$ to $\Sigma\sub{t+1}\spr{\star}$. 
So, we establish the contradiction by showing that the values of $\bfM, L$ that optimize problem in \eqref{eq:prop2helpersdp} have to satisfy $\bfM - L\Sigma\sub{t}\spr{-1}L\transpose = 0$.
Multiplying both sides of \eqref{eq:helpersdp-constr1} by $A\sub{t}\spr{-1}$ from left and $A\sub{t}\spr{-\mathrm{T}}$ from the right, we obtain:
\begin{subequations}\label{eq:prop2helpersdp2}
\begin{align}
    \min\sub{\substack{L \in \R{m\times n} \\\bfM \in \mathbb{S}\sub{m}\spr{+}}} & ~~ \tr{R\sub{n} \bfM} \\
    \text{s.t.} & ~~ Z =  L\transpose Y\transpose + Y L + Y \mathbf{M} Y\transpose\\
    & ~~ \eqref{eq:helpersdp-constr1} \nonumber
\end{align}
where $ Z = A\sub{t}\spr{-1} ( \Sigma\sub{t+1}\spr{\star} - \mathbf{H}\sub{t} )  A\sub{t}\spr{-\mathrm{T}}$, $Y = A\sub{t}\spr{-1} B\sub{t}$.
\end{subequations}
It is shown in \cite[Theorem 3]{p:balci2022exactcovariancewasserstein} that the SDP in \eqref{eq:prop2helpersdp2} admits a solution that satisfies $M - L \Sigma\spr{-1}\sub{t} L\transpose = \bm{0}$ if $R\sub{t} \succ \bm{0}$ which contradicts with our initial assumption.
This completes the proof.
\end{proof}

\begin{remark}
Note that, the assumption that $A\sub{k}$ is non-singular is not restrictive.
This is because in practice, $A\sub{k}$ is computed as the state transition matrix between discrete time steps of a continuous-time linear dynamical system \cite{b:rugh1996linear} for instance, in the case of a time-invariant system.
\end{remark}

\begin{remark}
The results that we obtained in Proposition \ref{prop:proposition2} coincides with the result in \cite{p:liu2022continuouscovmultiplicative} where the authors show that the optimal policy corresponds to a deterministic state feedback policy under the state multiplicative noise for a continuous-time linear system. 
\end{remark}

Even if the condition in Proposition \ref{prop:proposition2} is not satisfied i.e. $\Bar{B}\sub{k, \ell} \neq \bm{0}$ for some $k, \ell$, the SDP constraint $\bfM\sub{k} \succeq L\sub{k} \Sigma\sub{k}\spr{-1} L\sub{k}\transpose$ is tight for all $k$ in our numerical experiments which is presented in Section \ref{s:numerical-experiments}.
Therefore, the condition presented in Proposition \ref{prop:proposition2} is not strict. 
However, we left establishing a stricter condition that would yield the SDP relaxations to be tight as a future work. 


\section{Numerical Experiments}\label{s:numerical-experiments}
In this section, we present the results of our numerical experiments.
All numerical experiments run on a Mac M1 with 8GB of RAM. 
We used the CVXPY \cite{p:diamond2016cvxpy} package to parse the SDPs and used MOSEK \cite{mosek} as the SDP solver.
Specifically, we consider a UAV path planning problem.
The UAV is modeled as a point mass with double integrator dynamics (it is assumed that the nonlinear dynamics of the UAV is handled by lower-level velocity controllers, which is one of the standard approaches in the relevant literature \cite{p:schouwenaars2001mippathplan, p:blackmore2011chanceconstrpath, p:okamoto2019pathplanning}).
The state and the control input are defined as $ x\sub{k} =  [p\sub{k}\spr{x}, p\sub{k}\spr{y}, v\sub{k}\spr{x}, v\sub{k}\spr{y}]\transpose \in \R{4}$ and $u\sub{k} = [a\sub{k}\spr{x}, a\sub{k}\spr{y}]\transpose$ respectively,
where $p, v, a$ denote the position, velocity, and acceleration of the UAV, respectively. 
The matrices that determine the UAV dynamics are given as: 
\begin{align*}
    A\sub{k} = \begin{bmatrix} I\sub{2} & \Delta t I\sub{2} \\ \bm{0} & I\sub{2} \end{bmatrix}, 
    B\sub{k} = \begin{bmatrix} \Delta t\spr{2}/2 \\ \Delta t \end{bmatrix},
    W\sub{k} = \begin{bmatrix} \bm{0} & \bm{0} \\ \bm{0} & 0.01I\sub{2} 
    \end{bmatrix}
\end{align*}
where $\sqrt{\Delta t} = 0.1$. The number of multiplicative noise processes is given as $M = 2$ for all $k$ and for both state and input multiplicative noise. 
Therefore, $\Bar{A}\sub{k, 1} = \bdiag{\bm{0}, A\sub{b, 1}}$, $\Bar{A}\sub{k, 2} = \bdiag{\bm{0}, A\sub{b, 2}}$, $\Bar{B}\sub{k, 1} = [\bm{0}\transpose, B\sub{b, 1}\transpose]\transpose$, $\Bar{B}\sub{k, 2} = [\bm{0}\transpose, B\sub{b, 2}\transpose]\transpose$
where
\begin{align*}
    A\sub{k, 1} &= \beta\sub{1} \sqrt{\Delta t} \begin{bmatrix} 1.0 & 0 \\ 0.5 & 0 \end{bmatrix}, &  A\sub{k, 2} &= \beta\sub{2} \sqrt{\Delta t} \begin{bmatrix} 0 & 0.5 \\0 & 1.0 \end{bmatrix}, \\
    B\sub{k, 1} &= \theta\sub{1} \sqrt{\Delta t} \begin{bmatrix} 1.0 & 0\\0.5 & 0 \end{bmatrix}, & B\sub{k, 2} & = \theta\sub{2} \sqrt{\Delta t} \begin{bmatrix} 0 & 0.5 \\ 0 & 1 \end{bmatrix} .
\end{align*}
Here $[\beta\sub{1}, \beta\sub{2}, \theta\sub{1}, \theta\sub{2}] = [0.1, 0.3, 0.1, 0.6]$ are the noise intensity parameters. 
It is worth mentioning that the assumption of state and control multiplicative noise for UAV path planning tasks is more relevant than the assumption of additive noise since it is harder to follow the reference trajectory that describes an aggressive, jerky maneuver for the low-level controllers. In particular, the disturbances that amplify the error between the desired and the actual speed and acceleration should be proportional to the magnitudes of the speed and the acceleration.

\begin{figure}[h]
    \begin{subfigure}{0.45\linewidth}
        \includegraphics[width=\linewidth]{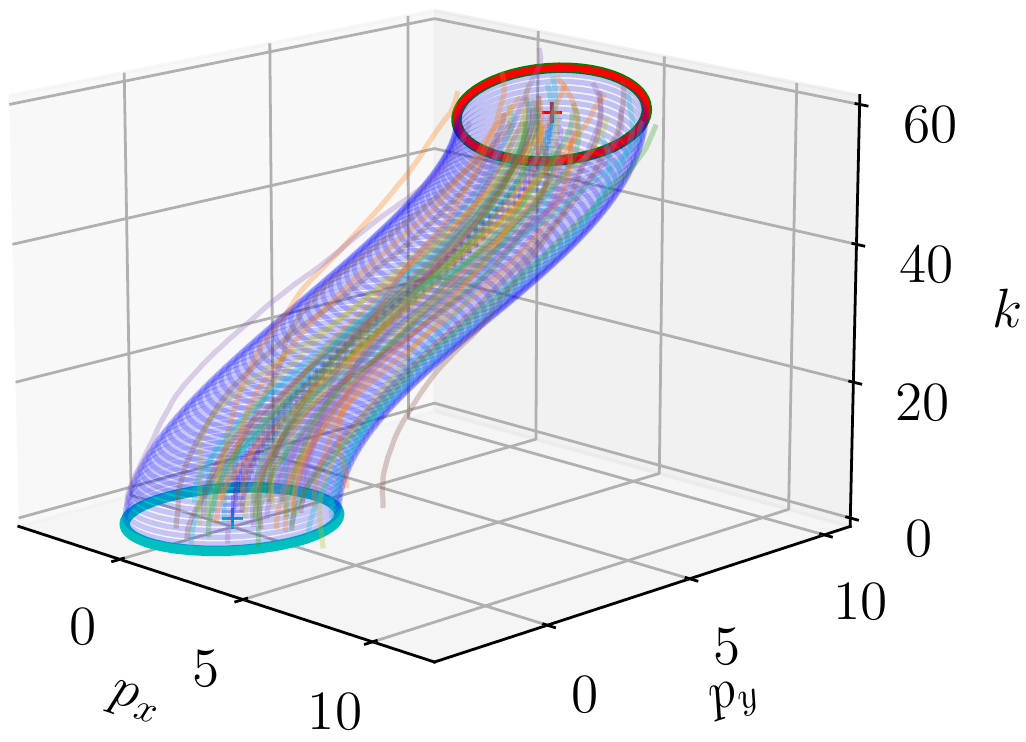}
        \subcaption{}\label{subfig:exact3duniform}
    \end{subfigure}
    \begin{subfigure}{0.45\linewidth}
        \includegraphics[width=\linewidth]{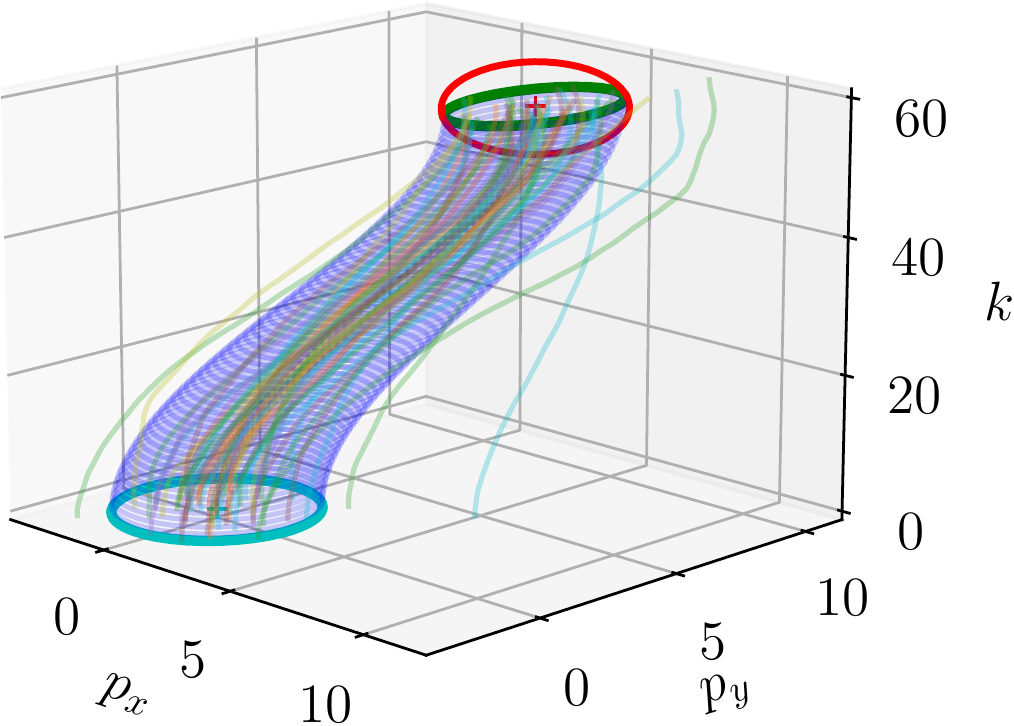}
        \subcaption{}\label{subfig:relaxed3duniform}
    \end{subfigure}
    \begin{subfigure}{0.45\linewidth}
        \includegraphics[width=\linewidth]{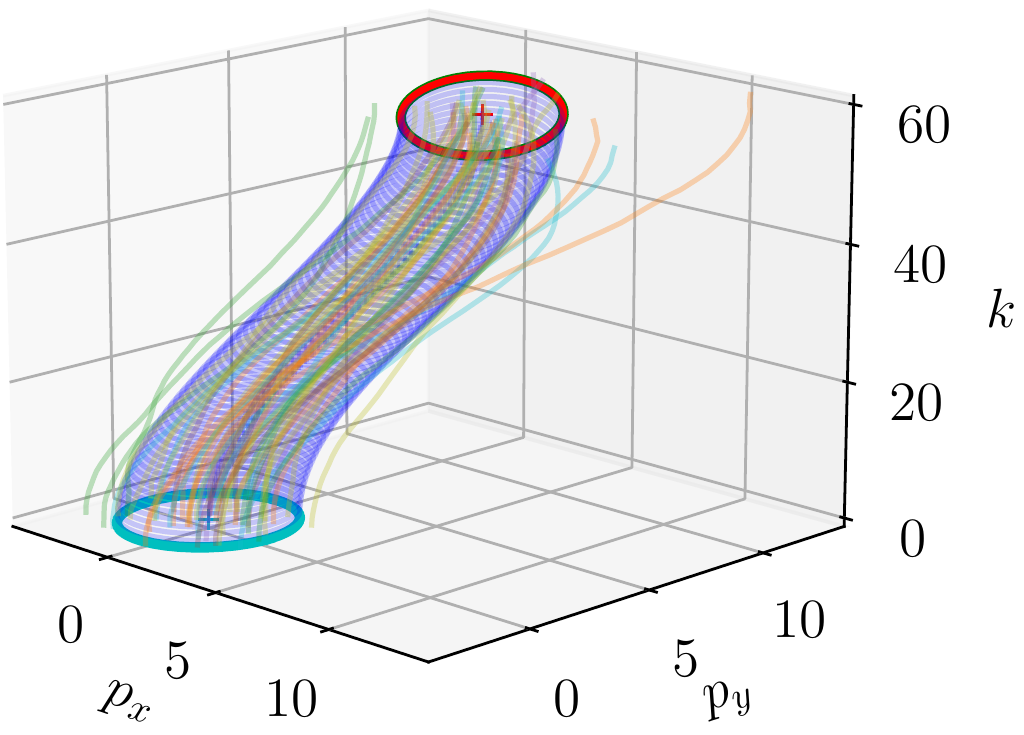}
        \subcaption{}\label{subfig:exact3dgaussian}
    \end{subfigure}
    \begin{subfigure}{0.45\linewidth}
        \includegraphics[width=\linewidth]{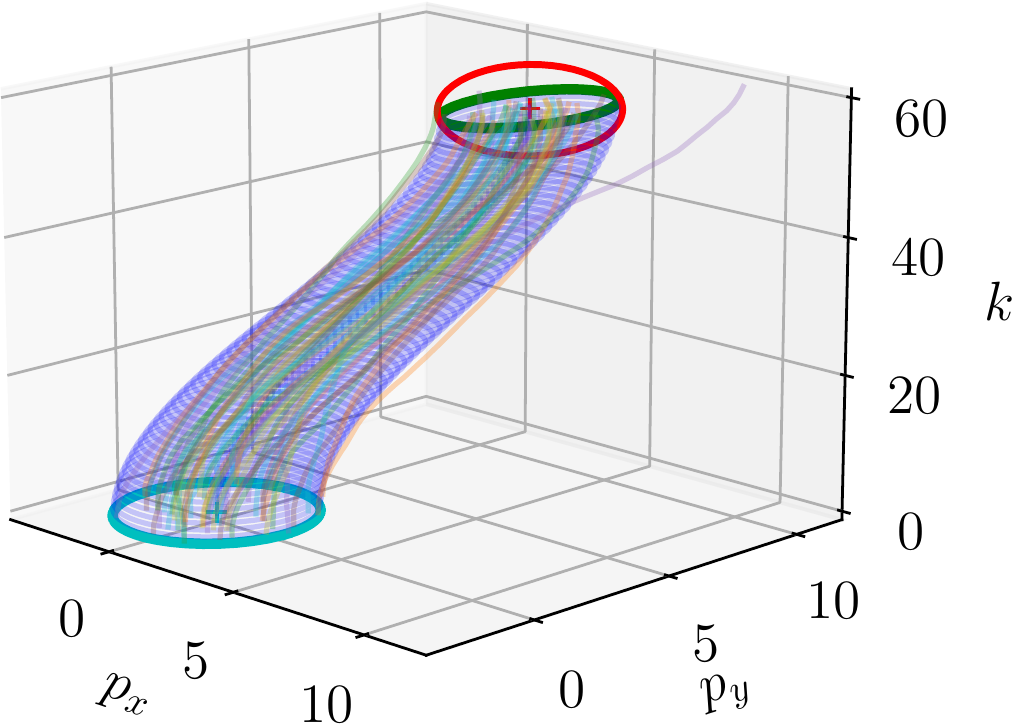}
        \subcaption{}\label{subfig:relaxed3dgaussian}
    \end{subfigure}
    \begin{subfigure}{0.45\linewidth}
        \includegraphics[width=\linewidth]{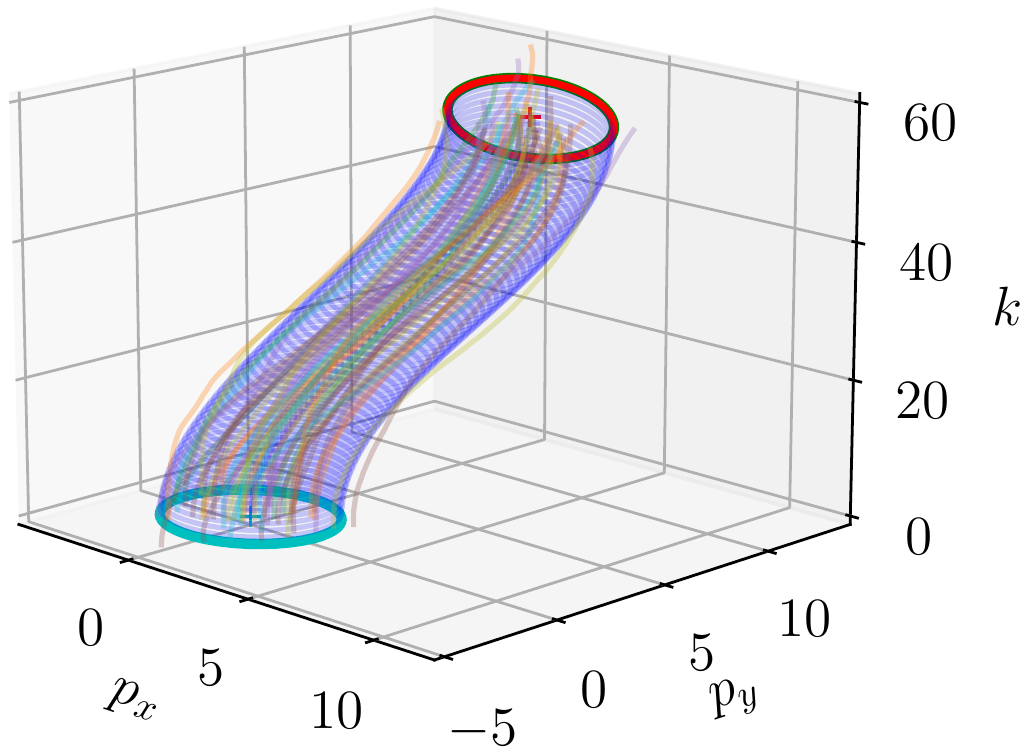}
        \subcaption{}\label{subfig:exact3ddiscrete}
    \end{subfigure}
    ~~~~
    \begin{subfigure}{0.45\linewidth}
        \includegraphics[width=\linewidth]{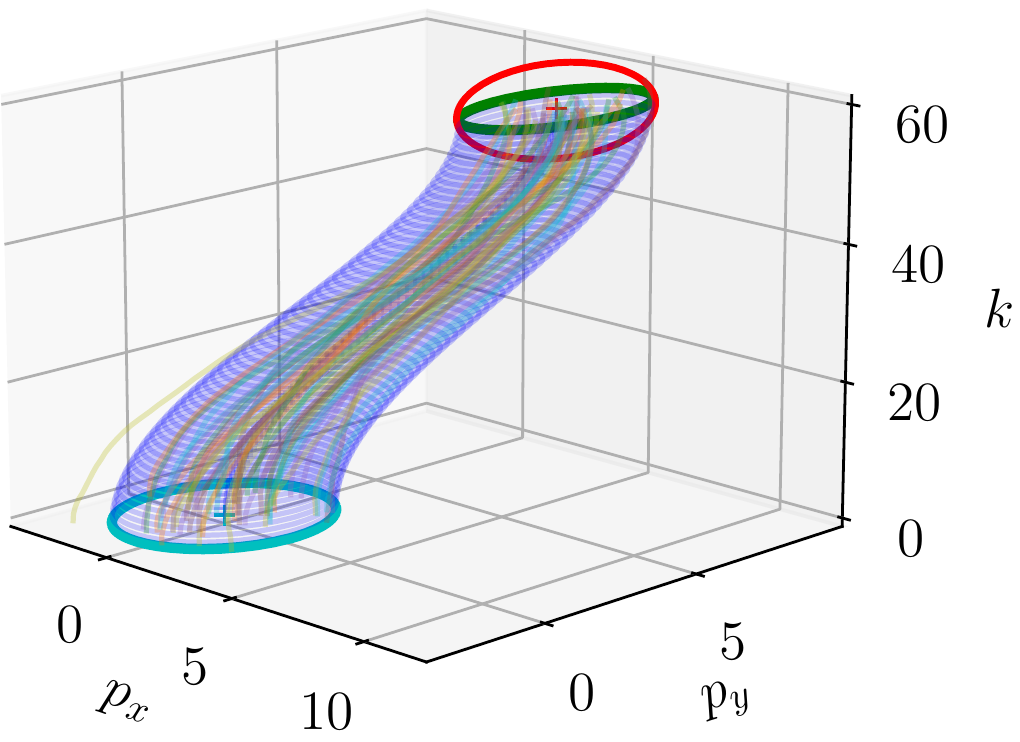}
        \subcaption{}\label{subfig:relaxed3ddiscrete}
    \end{subfigure}
    \caption{Evolution of the state statistics and sample trajectories under different noise distributions. The figures on the left (Figure \ref{subfig:exact3duniform}, \ref{subfig:exact3dgaussian} \ref{subfig:exact3ddiscrete}) show the results of the exact CS problem whereas the ones on the right (Figures \ref{subfig:relaxed3duniform}, \ref{subfig:relaxed3dgaussian}, \ref{subfig:relaxed3ddiscrete}) correspond to the relaxed CS problem. 
    Cyan, red and green ellipsoids correspond to the 2-$\sigma$ confidence ellipsoids of initial, desired, and terminal covariances, respectively. 
    The multiplicative noise terms $\delta\sub{k,\ell}, \gamma\sub{k,\ell}$ have uniform distribution over $[-\sqrt{3}, \sqrt{3}]$ in Figures \ref{subfig:exact3duniform}, \ref{subfig:relaxed3duniform}, unit normal distribution in Figures \ref{subfig:exact3dgaussian}, \ref{subfig:relaxed3dgaussian} and uniform distribution over $\{ -\sqrt{1.5}, 0, \sqrt{1.5} \}$ in Figures \ref{subfig:exact3ddiscrete}, \ref{subfig:relaxed3ddiscrete}. }\label{fig:3dplot}
\end{figure}

\begin{figure}[h]
    \centering
    \begin{subfigure}{0.45\linewidth}
        \includegraphics[width=\linewidth]{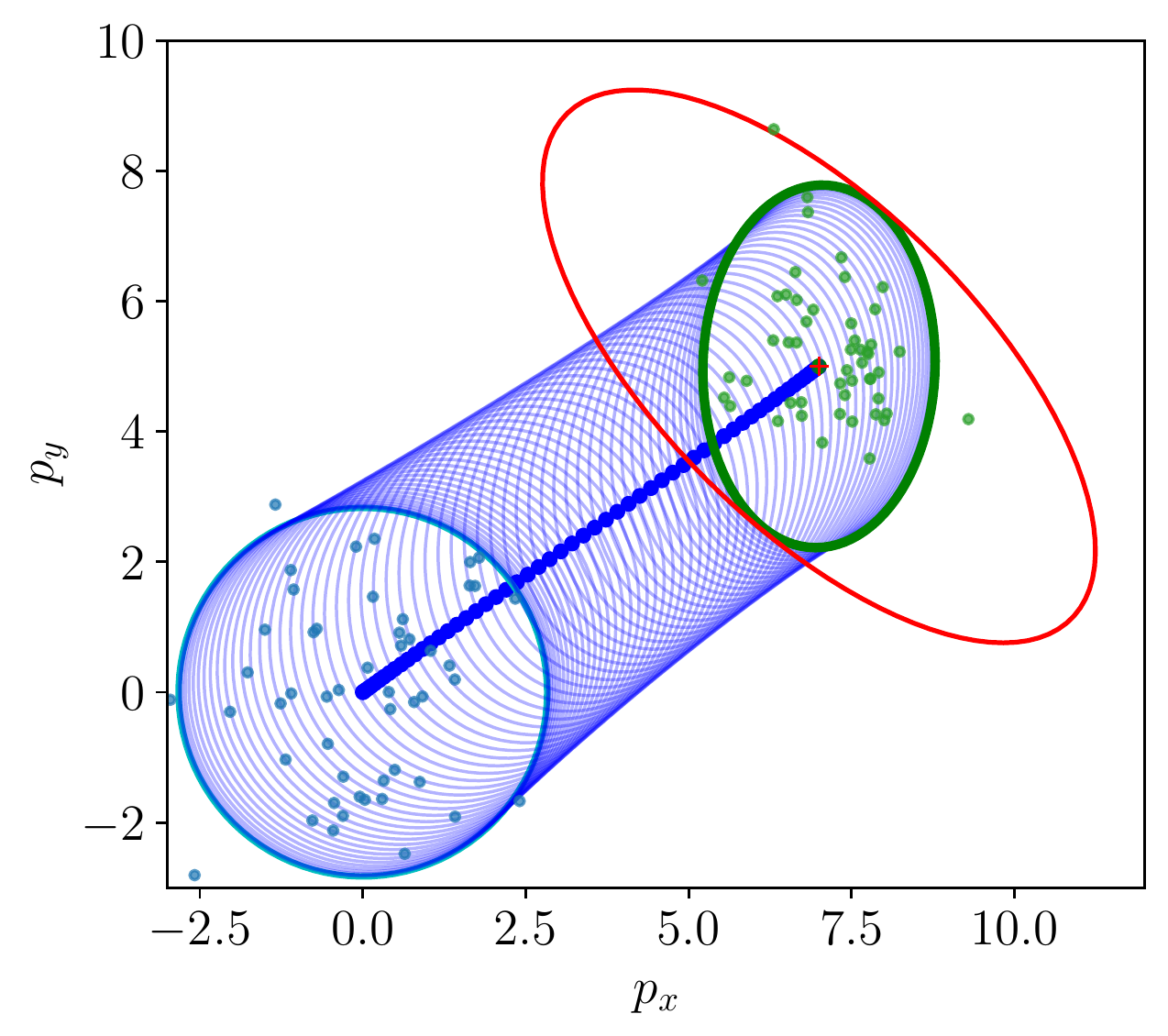}
        \subcaption{}\label{subfig:relaxed2ddiscrete}
    \end{subfigure}
    ~
    \begin{subfigure}{0.45\linewidth}
        \includegraphics[width=\linewidth]{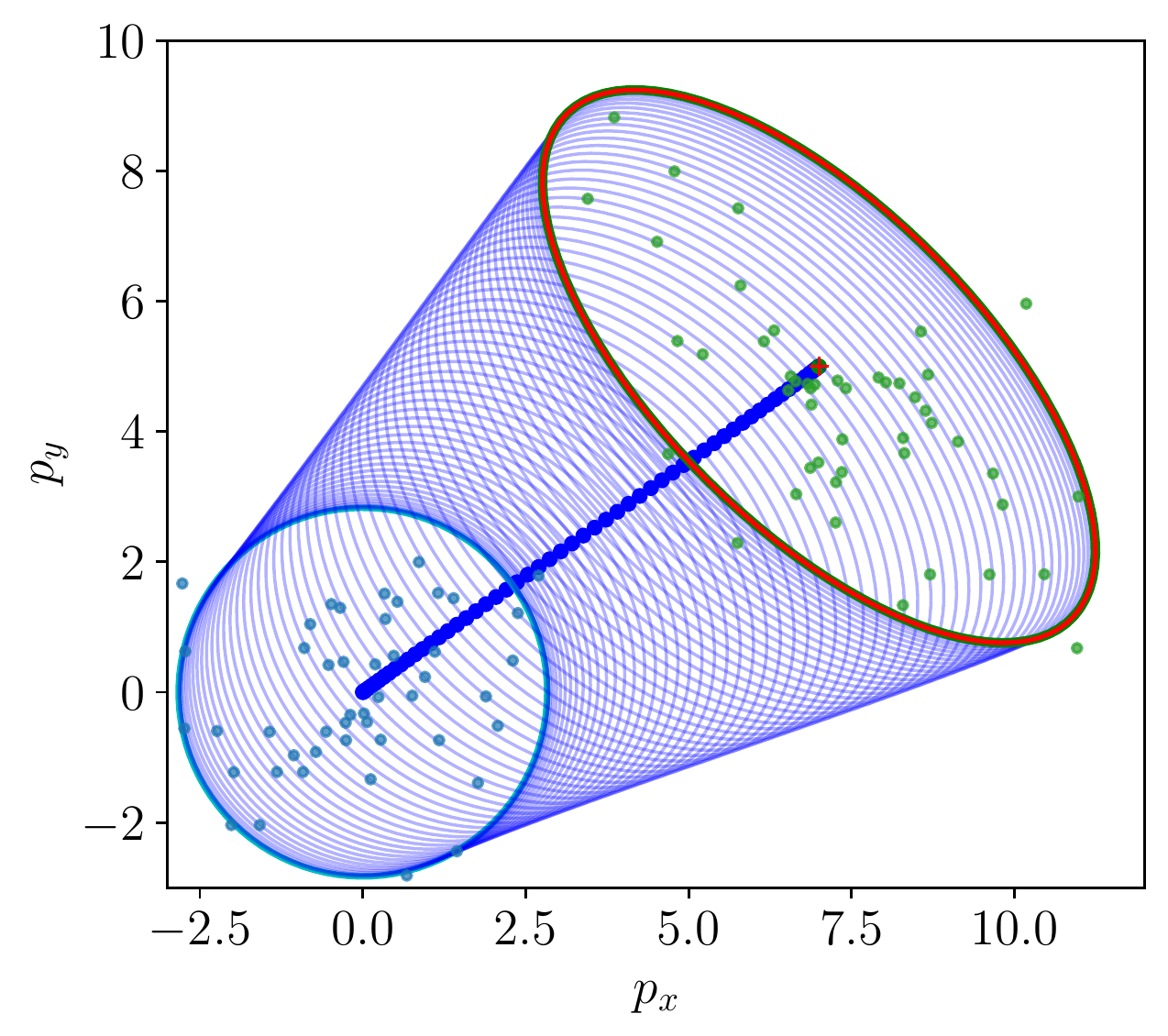}
        \subcaption{}\label{subfig:exact2ddiscrete}
    \end{subfigure}
    ~
    \begin{subfigure}{0.45\linewidth}
        \includegraphics[width=\linewidth]{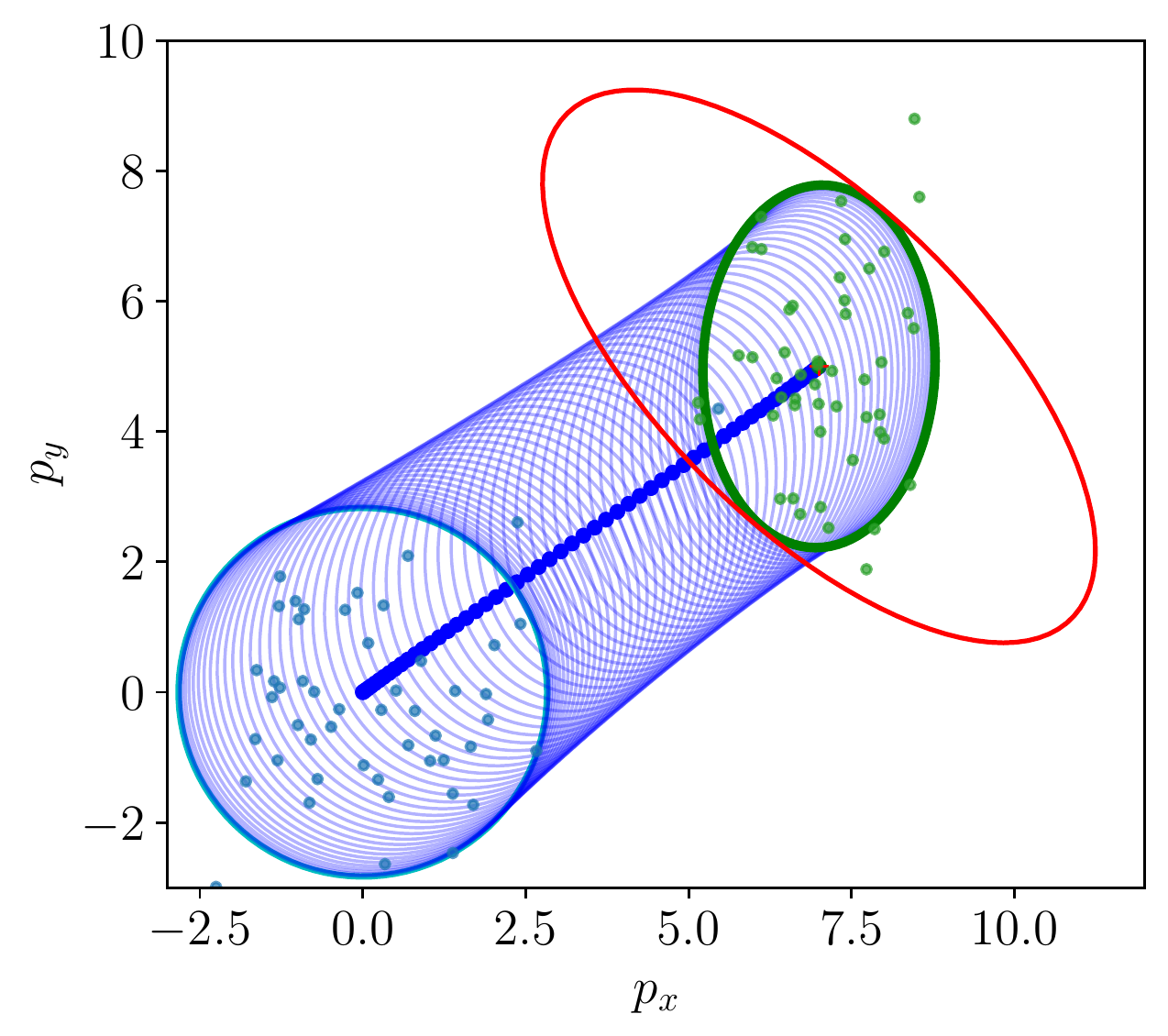}
        \subcaption{}\label{subfig:relaxed2duniform}
    \end{subfigure}
    ~
    \begin{subfigure}{0.45\linewidth}
        \includegraphics[width=\linewidth]{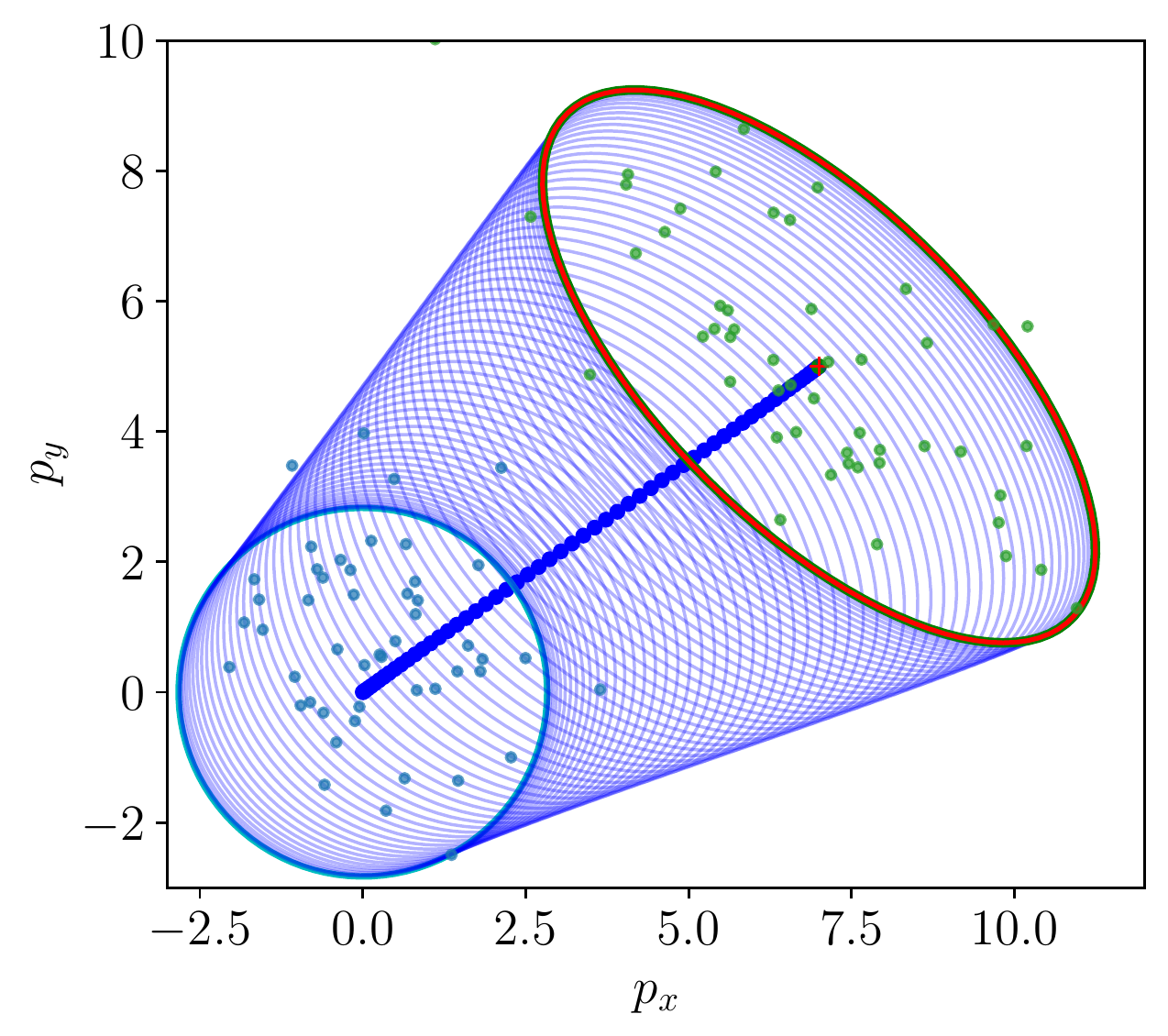}
        \subcaption{}\label{subfig:exact2duniform}
    \end{subfigure}
    \caption{Evolution of state statistics along with samples from initial and terminal distributions. Samples from the initial and the terminal distributions are illustrated by blue and green dots, respectively.  }\label{fig:plot2d}
\end{figure}

The initial state distribution of the UAV is a multivariate Gaussian with zero mean and covariance matrix $\Sigma\sub{0} = \bdiag{2.0 I\sub{2}, 0.01 I\sub{2}}$ whereas the desired mean and covariance matrix are given as $\mu\sub{\mathrm{d}} = [7.0, 5.0, 0.0, 0.0]\transpose$, $\Sigma\sub{\mathrm{d}} = \bdiag{\begin{smat} 4.5 & -3.0 \\ -3.0 & 4.5 \end{smat}, 0.1 I\sub{2}}$. Finally, the problem horizon is given as $N = 60$.

In Figure \ref{fig:3dplot}, the evolution of the state mean and covariance together with sample trajectories of the UAV dynamics under the control policies obtained by solving both the exact and the relaxed CS problems are presented.
In Figure \ref{fig:plot2d}, we illustrate the initial and terminal covariance matrices along with samples from the initial and terminal distributions. Note that the state of the UAV is modeled as a 2-dimensional (vector) double integrator (4 states). However, we only show the distribution of the position in the $x-y$ plane. In both Figure \ref{fig:3dplot} and \ref{fig:plot2d}, we sample 80 trajectories. It can be seen that terminal covariance constraints are satisfied for different multiplicative noise distributions.  

\section{Conclusion}\label{s:conclusion}
In this paper, we have addressed the exact and relaxed versions of the CS problem for discrete-time linear systems subject to mixed additive and multiplicative noise. We first recast the relaxed CS problem as a convex SDP. 
Then, we proposed a two-step solution method which leverages the solution to the relaxed CS problem to solve the exact CS problem. Finally, we gave an example that shows the necessity of randomized policies for the exact CS problem and provided a condition that guarantees the set of deterministic policies is sufficiently rich to address the latter problem. 
We also demonstrated in our numerical simulations, however, that the optimal policy may turn out to be deterministic, even when the provided condition is violated. 

In our future work, we plan to establish whether the solution procedure provided for the exact CS problem returns the globally optimal solution or not and provide a stronger condition that would guarantee that the deterministic policies can form a set of policies that is sufficiently rich for optimality.

\normalem
\bibliographystyle{ieeetr}
\bibliography{multCS.bib}

\end{document}